\documentclass[11pt]{article}
\usepackage{graphicx}
\usepackage[all]{xy}
\usepackage{amsmath,amsfonts,amsthm}
\usepackage{amssymb,latexsym,mathabx}
\usepackage{fixmath}
\usepackage{mathrsfs,amsbsy}
\usepackage{dsfont}

\usepackage{epsfig}

\newtheorem{theorem}{Theorem}[section]

\newtheorem{proposition}[theorem]{Proposition}

\newtheorem{lemma}[theorem]{Lemma}

\def\qed{\hfill $\Box$\medskip}

\def\IC{{\mathbb C}}
\def\IR{{\mathbb R}}

\def\a{{\alpha}}
\def\b{{\beta}}
\def\cH{{\cal H}}

\def\Ra{{\Rightarrow}}
\def\[{\left [}
\def\]{\right ]}
\def\({\left (}
\def\){\right )}

\def\Ra{{\ \Rightarrow\ }}

\def\Lra{{\ \Leftrightarrow\ }}
\def\dfrac{\displaystyle\frac}

\def\IC{{\mathbb C}}
\def\IR{{\mathbb R}}

\def\cH{{\cal H}}

\def\Ra{{\Rightarrow}}
\def\[{\left [}
\def\]{\right ]}
\def\({\left (}
\def\){\right )}

\def\Ra{{\ \Rightarrow\ }}

\def\Lra{{\ \Leftrightarrow\ }}
\def\dfrac{\displaystyle\frac}

\begin{document}
\openup 1\jot
\title{
Perturbing eigenvalues of non-negative matrices
}

\author{Chi-Kwong Li${}^{1}$, Yiu-Tung Poon${}^{2}$, and Xuefeng Wang${}^{3}$ \\
${}^{1}$Department of Mathematics, College of William and Mary,  \\
Williamsburg, VA 23187, USA. (ckli@math.wm.edu)\\
${}^{2}$Department of Mathematics, Iowa State University, \\
Ames, IA 50011, USA. (ytpoon@iastate.edu) \\
${}^{3}$Department of Mathematics,
Ocean University of China, \\
Qingdao,
Shandong 266100, China. (wangxuefeng@ouc.edu.cn)}

\date{}

\maketitle

\begin{center}
{\bf Dedicated to Professor Hans Schneider.}
\end{center}

\begin{abstract}
Let $A$ be an irreducible (entrywise) nonnegative $n\times n$ matrix with
eigenvalues
$$\rho, b+ic,b-ic, \lambda_4,\cdots,\lambda_n,$$
where  $\rho$ is the Perron eigenvalue.
It is shown that for any $t \in [0, \infty)$
there is a nonnegative matrix with eigenvalues
$$\rho+ \tilde t,\lambda_2+t,\lambda_3+t, \lambda_4 \cdots,\lambda_n,$$
whenever $\tilde t \ge \gamma_n t$ with
$\gamma_3=1, \gamma_4 = 2, \gamma_5=\sqrt 5$ and
$\gamma_n = 2.25$ for $n \ge 6$.
The result improves that of Guo et al. Our proof
depends on an auxiliary result in geometry asserting
that the area of an $n$-sided convex polygon is
bounded by $\gamma_n$ times the maximum area of the triangle
lying inside the polygon.
\end{abstract}

{\small

}

\section{Introduction}

The {\em nonnegative inverse eigenvalue problem} concerns the
study of necessary and sufficient conditions for a given set of
complex numbers $\lambda_1, \dots, \lambda_n$ to be the eigenvalues of
an (entrywise) nonnegative matrix.
This problem has attracted the attention of many authors, and is still open;
for example, see \cite{Eg} and its references.
In connection to this study, researchers study the change of the Perron eigenvalue
under the perturbation of the other real or complex eigenvalues of a given
nonnegative matrix. Here are several results in this direction.

\begin{itemize}
\item[1.] In \cite{Guo:97}, the author proved the following:

{\it Suppose $\rho,\lambda_2,\lambda_3,\cdots,\lambda_n$ are the eigenvalues
of an $n\times n$ nonnegative matrix $A$ such that
$\rho$ is the Perron eigenvalue, and  $\lambda_2$ is real.
Then for any $0 \le t \le \tilde t$,
there is a nonnegative matrix with eigenvalues
$\rho+ \tilde t,\lambda_2\pm t,\lambda_3,\cdots,\lambda_n$.}

\item[2.] Laffey \cite{Laffey:05} and Guo et al. \cite{GuoGuo:07}
obtained the following
independently:

{\it Suppose $\rho,\lambda_2,\lambda_3,\cdots,\lambda_n$ are the eigenvalues
of an $n\times n$ nonnegative matrix $A$ such that
$\rho$ is the Perron eigenvalue,
and $(\lambda_2, \lambda_3) = (b+ic,b-ic)$ is a
(non-real) complex conjugate pair.
Then for any $\tilde t, t \in [0, \infty)$
with $2t \le \tilde t$,
there is a nonnegative matrix with eigenvalues
$\rho+ \tilde t,\lambda_2- t,\lambda_3-t, \lambda_4 \cdots,\lambda_n$.}

\item[3.]
In \cite[Proposition 3.1]{GuoGuo:07}, the authors further showed that:

{\it Suppose $\rho,\lambda_2,\lambda_3,\cdots,\lambda_n$ are the eigenvalues of
an $n\times n$ nonnegative matrix $A$ such that
$\rho$ is the Perron eigenvalue,
and $(\lambda_2, \lambda_3) = (b+ic,b-ic)$ is a
(non-real) complex conjugate pair.
Then for any $\tilde t, t \in [0, \infty)$
with $4t \le \tilde t$,
there is a nonnegative matrix with eigenvalues
$\rho+\tilde t,\lambda_2+t,\lambda_3+t, \lambda_4 \cdots,\lambda_n$.}
\end{itemize}

\medskip\noindent
The results in (1) and (2) above were shown to be optimal in the sense that
the conclusion may fail if $\tilde t < t$ in (1) and $\tilde t < 2t$ in (2).
However, the result in (3) may be strengthened.
In this paper, we improve the third result, and prove the following.

\begin{theorem}\label{thm:MainResult}
 Suppose
 $\rho,\lambda_2,\lambda_3,\cdots,\lambda_n$ are the eigenvalues of
an $n\times n$ nonnegative matrix $A$ such that
$\rho$ is the Perron eigenvalue, and  $\lambda_2 = b+ic$ and
$\lambda_3 = b-ic$ are (non-real) complex conjugate pairs.
Then for any $t \in [0, \infty)$
there is a nonnegative matrix with eigenvalues
$$\rho+ \tilde t,\lambda_2+t,\lambda_3+t, \lambda_4 \cdots,\lambda_n,$$
whenever $\tilde t \ge \gamma_n t$ with
$\gamma_3=1, \gamma_4 = 2, \gamma_5=\sqrt 5$ and
$\gamma_n = 2.25$ for $n \ge 6$.
\end{theorem}

Our proof depends on the following geometrical result, which is of independent
interest.

\begin{proposition} \label{Hexagon} Suppose $n\in \{3,4,5,6\}$.
The area of an $n$-sided convex hexagon ${\cal P} \subseteq \IR^2$
is bounded by $\gamma_n$ times
the maximum area of the triangles lying inside ${\cal P}$, where

\medskip\centerline{
$\gamma_3=1, \ \gamma_4 = 2, \ \gamma_5=\sqrt 5, \ \gamma_6 = 2.25$,}

\medskip\noindent
and these bounds are best possible.
\end{proposition}

One easily sees that the maximum area of the triangles lying inside
a convex polygon is attained at a triangle formed by
3 of the vertices of the polygon.

The proof of Theorem \ref{thm:MainResult}
is given in Section 2, and the technical
proof of Proposition \ref{Hexagon} and some remarks are
given in Section 3.

%We will denote by $e$ the vector of order $n$ with all entries are equal to 1.

\section{Proof of Theorem \ref{thm:MainResult}}

\iffalse
%\begin{definition}\label{def:SotoNegativityIndex04}
Let $A=(a_{ij})_{i,j=1}^n$ be a real $n$ by $n$ matrix. We define the
{\it negativity index} of the $j$th
column $A_j$ of $A$ as
$$
\chi(A_j)\equiv \max \{0,-a_{1j},\cdots,-a_{nj}\}.
$$
We define the {\it negativity index} of $A$ as
$$
\chi(A)=\sum_{j=1}^n \chi(A_j).
$$
%\end{definition}
Notice that a real matrix $A$ is nonnegative if and only if
$\chi(A)=0.$

The set of all real matrices with all row sums equal to $\alpha$
is denoted by $\Omega_\alpha$.
Denote by $\sigma(A)$ the (multi-)set of eigenvalues of a matrix $A$.
A (multi-)set $S$ of complex numbers
is {\it realizable} if there
is a nonnegative matrix $A$ satisfying $\sigma(A) = S$.
The following result will be used;  see \cite{Johnson:81}.
\fi

We begin with two lemmas. The first one can be found in \cite{Johnson:81}.

\begin{lemma}\label{lem:Johnson81}
Suppose $\lambda_1, \dots, \lambda_n$ are the eigenvalues of
a nonnegative matrix. Then there is  a nonnegative matrix with constant
row sums with eigenvalues $\lambda_1, \dots, \lambda_n$.
\end{lemma}

The next lemma concerns the change of $r$ eigenvalues,
$\lambda_1,\dots,\lambda_r$ with $r < n,$ and leaving invariant the
other eigenvalues
of an $n\times n$ matrix $A$ by a rank-$r$ perturbation.
It can be viewed as an extension of the result in \cite{Perfect:55};
see also \cite[Theorems 27 and 33]{Brauer:52}.

\begin{lemma}\label{lem:ExtendPerfect55}
Let $A\in \IC^{n\times n}$
with eigenvalues $\lambda_1,\cdots,\lambda_n$.
Let $X=[x_1|x_2|\cdots|x_r]\in \IC^{n\times r}$ be such that
{\rm rank}$(X)=r$ and $AX=XD,$
where $D\in \IC^{r\times r}$
with eigenvalues $\lambda_1,\cdots,\lambda_r.$
Then for any
$r\times n$ matrix $C$, the matrix
$A+XC$ has eigenvalues $\mu_1,\cdots,\mu_r,
\lambda_{r+1},\cdots,\lambda_n$,
where $\mu_1,\cdots,\mu_r$ are eigenvalues of the matrix $D+CX$.
\end{lemma}
\begin{proof}
Let $S=[X|Y]$ be a nonsingular matrix with
$S^{-1}=\left[\begin{array}{c}
U\\ V
\end{array}\right] $, with $U\in \IC^{r\times n}$.
Then $UX=I_r,VY=I_{n-r}$ and $(VX)^t=UY=O_{r\times (n-r)}.$
\iffalse
Let $C=[C_1|C_2] ,\
 X=\left[\begin{array}{c}
X_1\\ X_2 \end{array}\right] , \
Y=\left[\begin{array}{c}
Y_1\\ Y_2 \end{array}\right]$,
with $C_1\in \IC^{n\times r}$, $C_2\in \IC^{n\times (n-r)}$,
$X_1\in \IC^{r\times r}$ , $X_2\in \IC^{(n-r)\times r}$ ,
$Y_1\in \IC^{r\times (n-r)}$ and , $Y_2\in \IC^{(n-r)\times (n-r)}$.
\fi
Because $AX=XD$, we have
\begin{equation}\label{eq:lemmaExtendPerfect55:1}
S^{-1}AS=\left[\begin{array}{c}
U\\
V
\end{array}\right]A[X,Y]=\left[\begin{array}{cc}
D & UAY\\
0 & VAY
\end{array}\right]
\end{equation}
and
\iffalse
$$
S^{-1}XCS=\left[\begin{array}{c}
I_r\\
0
\end{array}\right][C_1,C_2]S=\left[\begin{array}{cc}
C_1 & C_2\\
0 & 0
\end{array}\right]\left[\begin{array}{cc}
X_1 & Y_1\\
X_2 & Y_2
\end{array}\right]=\left[\begin{array}{cc}
CX & CY\\
0 & 0
\end{array}\right].
$$
\fi
$$
S^{-1}XCS=\left[\begin{array}{c}
I_r\\
0
\end{array}\right]CS=\left[\begin{array}{c}
C\\
0
\end{array}\right] [
X | Y  ]=\left[\begin{array}{cc}
CX & CY\\
0 & 0
\end{array}\right].
$$
Thus,
$$
S^{-1}(A+XC)S=S^{-1}AS+S^{-1}XCS=\left[\begin{array}{cc}
D+CX & UAY+CY\\
0    & VAY
\end{array}\right].
$$
Now, from \eqref{eq:lemmaExtendPerfect55:1}
we have $\sigma(VAY)=\{\lambda_{r+1},\cdots,\lambda_n \}$
and therefore
$$
\sigma(A+XC)=\sigma(D+CX) \cup \{\lambda_{r+1},\cdots,\lambda_n \}.
$$
\vskip -.3in \end{proof}

\noindent
{\bf We are now ready to present
the proof of Theorem \ref{thm:MainResult}.}

Let $A\in \Omega_\rho$ be an $n\times n$ non-negative real matrix with
eigenvalues
$\rho,b+ic,b-ic,\lambda_4,\cdots,\lambda_n$,
and let
$u\pm iv$ are eigenvectors of $A$ corresponding to the eigenvalues $b\pm ic$, where
$u=(u_1,u_2,\cdots,u_n)^T, v=(v_1,v_2,\cdots,v_n)^T\in \IR^n$.
Then we have the following equality for $n\times 2$ matrices:
\begin{equation}\label{eq:them:3.1:0}
A[u|v]=[u|v]\left[\begin{array}{cc}
b  & c \\
-c & b \\
\end{array}\right].
\end{equation}
We adopt an idea in \cite{GuoGuo:07}
and let
$$M = \begin{bmatrix}1 & \cdots & 1 \cr u_1 & \cdots & u_n \cr v_1 & \cdots & v_n\cr
\end{bmatrix}.$$
Denote by $P = P(u,v)$ a point in $\IR^2$ with co-ordinate $(u,v)$.
\iffalse
Suppose $n$ vertices of a convex polygon are given:
$$P_1(x_1,y_1), \dots, P_n(x_n,y_n)$$
arranged in counterclockwise direction.
By Analytic Geometry,
\begin{equation}\label{equ:area:det}
|\det(i,j,k)|=\left|\det\left(
\begin{array}{ccc}
1   & 1   & 1  \\
u_i & u_j & u_k\\
v_i & v_j & v_k\\
\end{array}\right)\right|
\end{equation}
is 2 times the area of triangle with vertices $P_i(u_i,v_i),\ P_j(u_j,v_j)$
and $P_k(u_k,v_k)$; $\det(i,j,k) > 0$ if and only if the corresponding
triangle has nonzero area and the points $P(u_i,v_i), P(u_j,v_j)$
and $P(u_k,v_k)$ are in counterclockwise direction.
\fi
By Analytic Geometry, suppose
$$\det(i,j,k) = \det\left(
\begin{bmatrix}1 & 1 & 1 \cr u_i & u_j & u_k  \cr v_i & v_j & v_k\cr
\end{bmatrix}\right), \qquad 1 \le i , j , k \le n.$$
\iffalse
Using the notation at the end of Section 1, we consider
the points $P_i(u_i,v_i), i=1,2,\cdots,n$ in $\IR^2$.
\fi
Then $|\det(i,j,k)|$ is 2 times the area of the triangle
with vertices $P_i(u_i,v_i),\ P_j(u_j,v_j)$
and $P_k(u_k,v_k)$. Moreover,
 $\det(i,j,k)>0$ if and only if the points
$P_i\rightarrow P_j\rightarrow P_k\rightarrow P_i$ are not collinear and
appear  in counterclockwise direction in $\IR^2$.

Replacing $(A,u,v)$ by $(QAQ^T,Qu,Qv)$ for a suitable permutation
matrix $Q$, we may assume that
\begin{equation}\label{eq:them:3.1:1}
\Delta=\det(1,2,3)=\max_{1\le i,j,k \le n} \det(i,j,k).
\end{equation}
Recall that $e = (1, \dots, 1)^T$.
Since $e,u+iv, u-iv$ are the eigenvectors of
the distinct eigenvalues $\rho, \lambda_2, \lambda_3$,
so $e,u,v$ are linearly independent over $\IR$.
It follows that $\Delta=\det(1,2,3)>0$.
Let
$$
x=(x_1,x_2,x_3,0,\cdots,0)^T \quad \hbox{ and }
\quad y=(y_1,y_2,y_3,0,\cdots,0)^T
$$
satisfy
\begin{equation}\label{eq:them:3.1:2}
x^Te=0, x^Tu=1, x^Tv=0; y^Te=0, y^Tu=0, y^Tv=1,
\end{equation}
that is,
$$\begin{bmatrix} 1 & 1 & 1 \cr u_1 & u_2 & u_3 \cr v_1 & v_2 & v_3\cr\end{bmatrix}
\begin{bmatrix} x_1 & y_1  \cr x_2 & y_2  \cr x_3 & y_3 \cr\end{bmatrix} =
\begin{bmatrix} 0 & 0  \cr 1 & 0 \cr 0 & 1\cr\end{bmatrix}.$$
Then
\begin{multline}\label{multline:them:3.1:0}
\ \hskip 1.1 in x_1=\frac{1}{\Delta}(v_2-v_3),x_2=\frac{1}{\Delta}(v_3-v_1),
x_3 =\frac{1}{\Delta}(v_1-v_2),\\
y_1 =\frac{1}{\Delta}(u_3-u_2),y_2
=\frac{1}{\Delta}(u_1-u_3),y_3=\frac{1}{\Delta}(u_2-u_1). \hskip .7in ~
\end{multline}
and
$$
[x,y]^T[u,v]=I_2.
$$
Suppose
\begin{equation}\label{eq:them:3.1:20}
[u|v][x|y]^T=\left[\begin{array}{cccccc}
\a_{11} & \a_{12} & \a_{13} & 0 &\cdots & 0\\
\a_{21} & \a_{22} & \a_{23} & 0 &\cdots & 0\\
\a_{31} & \a_{32} & \a_{33} & 0 &\cdots & 0\\
\vdots & \vdots & \vdots &\vdots &   &\vdots\\
\a_{n1} &\a_{n2} & \a_{n3} & 0 &\cdots & 0\\
\end{array}\right].
\end{equation}
Then for $i = 1, \dots, n$,
$$\a_{i1}  = u_ix_1 + v_iy_1 =
\frac{1}{\Delta}\det(i,2,3)-\frac{1}{\Delta}(u_2v_3-u_3v_2),$$
$$ \a_{i2} = u_ix_2+v_i y_2 =
\frac{1}{\Delta}\det(1,i,3)-\frac{1}{\Delta}(u_3v_1-u_1v_3),$$
$$\a_{i3} = u_ix_3+v_iy_3 =
\frac{1}{\Delta}\det(1,2,i)-\frac{1}{\Delta}(u_1v_2-u_2v_1).$$
If
$$
c_{i1}=\a_{i1}-\a_{21} = \frac{1}{\Delta}\det(i,2,3), \
c_{i2}=\a_{i2}-\a_{32} = \frac{1}{\Delta}\det(1,i,3), \
c_{i3}=\a_{i3}-\a_{23} = \frac{1}{\Delta}\det(1,2,i),$$
then
\begin{equation}\label{eq:them:3.1:4}
c_{11}\ge c_{i1}, \quad c_{22}\ge c_{i2}, \quad c_{33}\ge c_{i3},
\end{equation}
because $\Delta= \det(1,2,3) \ge \det(i,j,k)$ for all $1 \le i,j,k \le n$.
Let
$$
c_{i1}=\min_{l=1,2\dots, n} c_{l1}, \quad
c_{j2}=\min_{l=1,2\dots, n} c_{l2}, \quad
c_{k3}=\min_{l=1,2\dots, n} c_{l3}.$$
Then $
c_{i1}\le c_{l 1},\ c_{j2}\le c_{l 2}\ \mbox{ and }\ c_{k3}\le c_{l 3}\ \mbox{ for all  }l=1,2\dots, n$. Therefore, we have
\begin{equation}
\label{ijk}
\alpha_{i1}\le \alpha_{l 1}, \ \alpha_{j2}\le \alpha_{l 2}\ \mbox{ and }\ \alpha_{k3}\le \alpha_{l 3}\ \mbox{ for all  }l=1,2\dots, n.
\end{equation}

Assume that $n \ge 6$, and that $1,2,3,i,j,k$ are distinct,
and focus on
\begin{equation}\label{tildeM}
\tilde M = \begin{bmatrix}
1 & 1 & 1 & 1 & 1 & 1 \cr
u_1& u_2 & u_3 & u_i & u_j & u_k \cr
v_1& v_2 & v_3 & v_i & v_j & v_k\cr
\end{bmatrix}.
\end{equation}
Note that for the following points in $\IR^2$,
$$P_1(u_1,v_1), P_2(u_2,v_2), P_3(u_3,v_3),
P_i(u_i,v_i), P_j(u_j,v_j), P_k(u_k,v_k),$$
\begin{itemize}
\item the area of a triangle  formed by any three of these points is not larger
than $\dfrac{\det(1,2,3)}{2}$, which is the area of the triangle with vertices
$P_1, P_2, P_3$;
\item $c_{i1} \le\frac{\det(2,2,3)}{\Delta}= \frac{\det(2,3,3)}{\Delta} = 0$, \
$c_{j2} \le\frac{\det(1,1,3)}{\Delta}= \frac{\det(1,3,3)}{\Delta} = 0$, \
$c_{k3} \le\frac{\det(1,2,1)}{\Delta}= \frac{\det(1,2,2)}{\Delta} = 0$.
\end{itemize}
Thus, $\det(i,2,3), \det(1,j,3), \det(1,2,k) \in (-\infty,0]$.
Note that $\det(r,s,t) \le 0$ if and only if $P_r, P_s, P_t$ are
collinear or they are in clockwise direction. Let $\ell_1$
(respectively, $\ell_2,\ \ell_3$) be the line through $P_1$
(respectively, $P_2,\ P_3$) parallel to $\overline{P_2P_3}$, (respectively,
$\overline{P_1P_3},\ \overline{P_1P_2}$). Suppose $\ell_2$ and $\ell_3$
(respectively,  $\ell_1$ and $\ell_3$, $\ell_1$ and $\ell_2$) intersect at $Q_1$
(respectively,  $Q_2$ and $Q_3$).
Since $\det(i,2,3)\le 0$ and $|\det(1,2,i)|,\ |\det(1,3,i)|\le \det(1,2,3)$, $P_i$
lies in the triangle $Q_1P_3P_2$. Similarly, $P_j$ and $P_k$ lie in the triangles
$P_1P_3Q_2$ and $P_1Q_3P_2$ respectively.  Thus
$P_1P_kP_2P_iP_3P_j$ is a convex hexagon (including
the degenerate cases, when it is a
triangle, quadrilateral or pentagon).
Moreover, the vertices $P_1, P_j, P_3, P_i, P_2, P_k, P_1$
are in clockwise direction.
By Proposition \ref{Hexagon},
$$\frac{5}{4} \ge \frac{1}{\Delta}\left(|\det(i,2,3)|
+ |\det(1,j,3)| + |\det(1,2,k)|\right)
= -( c_{i1}+c_{j2}+c_{k3})\ge 0.
$$
It follows that
\begin{equation}\label{equ:max:maintheo}
-1\ge \a_{i1}+\a_{j2}+\a_{k3} =
c_{i1}+\a_{21} + c_{j2} + \a_{32} + c_{k3} + \a_{23} \ge
-\frac{5}{4}-1=-2.25.
\end{equation}
 Suppose $\tilde t\ge 2.25t\ge 0$. Let
$$\delta=\dfrac{\tilde t+t(\a_{i1} + \a_{j2} + \a_{k3})}{3}\ge\dfrac{\tilde t-2.25t}3\ge 0$$
Set
$$z = (- t\a_{i1}+\delta,\ - t\a_{j2}+\delta,\ - t\a_{k3}+\delta,0,\cdots,0)^T\ \mbox{ and }\tilde A = A+[e|u|v][z|tx|ty]^T.$$
By direct computation, we have
$$[z|tx|ty]^T[e|u|v]=\[\begin{array}{ccc} \tilde t&*&*\\
0&t&0\\
0&0&t\end{array}\]\,.$$
By Lemma \ref{lem:ExtendPerfect55},
the eigenvalues of $\tilde A$ are
$\rho+\tilde t,\sigma_2,\sigma_3,\lambda_4,
\cdots,\lambda_n,$ where $\sigma_2,\sigma_3$
are the eigenvalues of
$\left[\begin{array}{cc}
b  & c\\
-c & b\\
\end{array}\right]+tI_2$, that is, $\sigma_2=b+t+ic,\ \sigma_3=b+t-ic$.

Let

$$[e|u|v][z|tx|ty]^T=\left[\begin{array}{cccccc}
\b_{11} & \b_{12} & \b_{13} & 0 &\cdots & 0\\
\b_{21} & \b_{22} & \b_{23} & 0 &\cdots & 0\\
\b_{31} & \b_{32} & \b_{33} & 0 &\cdots & 0\\
\vdots & \vdots & \vdots &\vdots &   &\vdots\\
\b_{n1} &\b_{n2} & \b_{n3} & 0 &\cdots & 0\\
\end{array}\right]$$
By (\ref{ijk}), we have
$$
\begin{array}{rl}
\b_{l1}=&t(\a_{l1}-\a_{i1})+\delta\ge 0\\&\\
\b_{l2}=&t(\a_{l2}-\a_{j2})+\delta\ge 0\\&\\
\b_{l3}=&t(\a_{l3}-\a_{k3})+\delta\ge 0.\end{array}$$
Thus, $\tilde A$ also has nonnegative entries.
Hence, $\tilde A$ is the desired matrix.

Suppose $n = 5,4,3$. Then the matrix $\tilde M$ in (\ref{tildeM}) has at most
$n$ columns. Nevertheless,
we can apply a similar argument and use the
corresponding result in Proposition \ref{Hexagon} to construct the desired
matrix $\tilde A$. We omit the details.
\qed

\section{Proof of Proposition \ref{Hexagon}}
\setcounter{equation}{0}

The purpose of this section is to prove the Proposition \ref{Hexagon}.
The results for $n = 3$ is trivial.

We will assume that $P_1,\dots,P_n$ are vertices of the convex polygon arranged in counterclockwise direction. The following two facts are useful in our discussion.

(a)  One can  apply an affine transformation $v \mapsto Tv + v_0$
for some invertible $2\times 2$ matrix $T$ and $v_0 \in \IR^2$ to the points
$P_1, \dots, P_n$ without affecting the hypothesis and conclusion of
the result.

(b) One can always find an affine map to
send any 3 vertices of the polygon  to any 3 non-collinear points.

\medskip
Suppose $n = 4$. One may apply an affine transformation and assume that
$P_1=(0,0),\ P_2=(1,0),\ P_3=(1,1)$  are the   vertices of the triangle of largest area.
Since  all the triangles inside the quadrilateral have area at most 1/2,
  the fourth vertex is in the triangle with vertices $(0,0), (1,1), (0,1)$.
The conclusion of Proposition \ref{Hexagon} follows readily.

\medskip
Suppose $n = 5$ and $P_1,\dots,P_5$ are vertices of a convex pentagon arranged in counterclockwise direction. Let $T$ be a triangle of largest area.

{\bf Case 1. } $T$ has two sides in common with the pentagon. We may assume that
$P_1=(0,0),\ P_2=(1,0),\ P_3=(1,1)$ are the vertices of $T$.
Then $P_4$ and $P_5$ have to lie in the triangle with vertices $(1,0), (1,1), (0,1)$
and the conclusion of Proposition \ref{Hexagon} follows readily.

{\bf Case 2. }  $T$ has only one side in common with the pentagon. We may assume that
$P_1=(0,0),\ P_2=(1,0),\ P_4=(0,1)$ are the vertices of $T$. Then we have

\begin{itemize}
\item[{\rm (a)}] $P_3=(u_3,v_3)$   lies in the triangle with vertices $(1,0), (1,1), (0,1)$, and
    \item[{\rm (b)}]     $P_5=(-u_5,v_5)$   lies  in the triangle with vertices $(0,0), (0,1), (-1,1)$.
\end{itemize}
By applying the affine transformation $(x,\ y)\mapsto (1-(x+y),\ y)$, if necessary,
we may assume that $v_3\ge v_5$. For the convenience of calculation, we will use
$\Delta(i,j,k)$ to denote twice the area of the triangle with vertices
$P_i,\ P_j,\ P_k$.  We will show that subject to the constraints (a), (b) and
$\Delta(2,3,5)\le 1$, we have $\Delta(1,2,4)+\Delta(2,3,4)+ \Delta(1,4,5)\le \sqrt{5}$,
where the equality holds at $(u_3,v_3) = (2, \sqrt 5-1)/2$
and $(-u_5,v_5) = (1-\sqrt 5, \sqrt 5 -1)/2$.

\medskip
By direct calculation, we have
$$\begin{array}{c}\Delta(2,3,5) =v_3(1+u_5)-(1-u_3)v_5\quad \mbox{ and }\\ \\
\Delta(1,2,4)+\Delta(2,3,4)+ \Delta(1,4,5)=u_3+u_5+v_3\,.\end{array}$$

So we need to show that  subject to the constraints
\begin{equation}\label{eq5}
%\begin{array}{c}
1\le u_3+v_3,\ u_3\le 1,\ 0\le u_5\le v_5\le v_3\le 1,\ %\\ \\
v_3(1+u_5)-(1-u_3)v_5\le 1\,,
%\end{array}
\end{equation}
the maximum value of $u_3+u_5+v_3$ is $\sqrt{5}$.

\smallskip

We can replace $v_5$ by $v_3$ without changing  $u_3+u_5+v_3$ or violating the constraints. So we will assume that $v_5=v_3$. Then the constraints in (\ref{eq5}) becomes

$$1\le u_3+v_3,\ u_3\le 1,\ 0\le u_5\le   v_3\le 1,\ (u_3+u_5)v_3\le 1$$
So we have $u_3+u_5 \le 1+ v_3,\ \dfrac{1}{v_3}$. Therefore, for fixed $0\le v_3 \le 1$, the maximum of $u_3+u_5+v_3$ is equal to $1+ 2v_3$, if $1+v_3 \le \dfrac{1}{v_3}\Lra v_3\le \dfrac{\sqrt{5}-1}2$, and $v_3+ \dfrac{1}{v_3}$ if $1+v_3 \le \dfrac{1}{v_3}\Lra v_3\ge \dfrac{\sqrt{5}-1}2$. Maximizing over $v_3$ in both cases, we have the maximum value $\sqrt{5}$   attained at $v_3= \dfrac{\sqrt{5}-1}2$. Thus the maximum of $u_3+u_5+v_3$ is attained at $u_3=1,\ u_5=v_3=v_5=\dfrac{\sqrt{5}-1}2$. We note that for these values of $u_3,\ u_5,\ v_3,\ v_5$, we actually have  $\Delta(i,j,k)\le 1 $ for all $1\le i<j<k\le 5$.

\bigskip

Finally, we consider the intricate case when $n = 6$.
Suppose a (non-degenerate) convex hexagon has vertices $P_1(x_1,y_1), \dots, P_6(x_6,y_6)$ arranged
in counterclockwise direction. We will prove that

\begin{equation}\label{max}
\dfrac{\mbox{Area of the hexagon with vertices }P_1,\ P_2, \dots, P_6}
{\max\{\mbox{Area of triangle  with vertices }  P_i,\  P_j ,\ P_k :
1\le i<j<k\le 6\}}\le \dfrac{9}{4},
\end{equation}
where the inequality becomes an equality
for the hexagon $\cH_0$ with vertices
$$(0,\ 0),\ (1,\ 0),\ (\dfrac{5}{6},\ \dfrac{2}{3}),\ (0,\ 1),
\ (-\dfrac{1}{4},\ 1),\ (-\dfrac{2}{3},\ \dfrac{2}{3}).$$
Note that a direct calculation
shows that the area of the  triangle with vertices
$(0,\ 0),\ (1,\ 0),\   (0,\ 1)$ is $\dfrac{1}{2}$,
which is  maximum among all triangles with vertices from $\cH_0$.

\medskip
\begin{lemma}\label{boundary}
Suppose the maximum of the left hand side of (\ref{max}) is attained at some
hexagon $\cH$ with vertices $P_1 , \dots, P_6 $. Then
$$ \max\{\mbox{Area of triangle  with vertices }
P_i,\  P_j ,\ P_k :1\le i<j<k\le 6\}$$
is attained at some  triangle with at least one side in common with the
boundary  of $\cH$.
\end{lemma}

\it Proof. \rm
 Let $M$ be the maximum of the left hand side of (\ref{max})
 over all (non-degenerate) convex hexagon.
Clearly, $M$ exists and $\dfrac{9}{4}\le M\le 4$.

Suppose the maximum of the left hand side of (\ref{max}) is attained at some hexagon
$\cH$ with vertices $P_1 , \dots, P_6 $, labeled in counterclockwise direction.
We are going to prove the result  by contradiction.

Suppose the maximum of the area of triangles with vertices $ P_i,\  P_j ,\ P_k$,
$1\le i<j<k\le 6$ can only be  attained at triangles with no side in common with the
hexagon $\cH$. Without loss of generality, we may assume that the maximum is attained
at the triangle with vertices $ P_1,\  P_3 ,\ P_5$.
Using an affine transformation, we may assume that $P_1=(0,\ 0)$, $P_3=(1,\ 0)$ and
$P_5=(0,\ 1)$. For the convenience of notation and computation, let
$$\Delta(i,j,k)=2\times (\mbox{area of triangle with vertices }P_i,\ P_j,\ P_k)$$
for $1\le i<j<k\le 6$.
By our assumption, we have
\begin{equation}\label{135}
\Delta(1,3,5)=1 ,\  \Delta(2,4,6)\le 1\ \mbox{ and }\Delta(i,j,k)<1\ \mbox{ for all }
(i,j,k)\neq (1,3,5),\ (2,4,6).
\end{equation}
We will prove that under the conditions in (\ref{135}), the area of the hexagon
$\cH$ is less than or equal to $1$, which contradicts the fact that
$M\ge \dfrac{9}{4}$ as shown by our example before Lemma \ref{boundary}.

In the following,
we will prove that under the conditions in (\ref{135}), we have
\begin{equation}\label{123}
\Delta_0=\Delta(1,2,3)+ \Delta(3,4,5)+ \Delta(1,5,6)\le 1
\end{equation}
Suppose $P_2=(u_1,\ -v_1)$, $P_4=(u_2,\ v_2)$ and $P_6=(-u_3,v_3)$.
Let
$$   A  =\left[\begin{array}{cccccc}1 & 1 & 1 & 1  & 1  & 1\\
0&u_1&1&u_2&0&-u_3\\
0&-v_1&0&v_2& 1&v_3\end{array}\right].
$$
Then  $|\Delta(i,j,k)|$ is equal to the determinant of the
submatrix of $A$ lying in columns $i,j,k$. By (\ref{135}), we have

\medskip
\begin{equation}\label{strict}
  \begin{array}{rcl}\Delta(1,3,5) & =
  & 1\mbox{  is the maximum, among all }\Delta(i,j,k)\\&\\
 \Delta(2,4,6)  &=&  (u_2- u_1)(v_1+ v_3) + (u_1   + u_3)(v_1 + v_2)
 \le 1 ,\ \mbox{ and }\\& &\\
0\le v_1<u_1<1,& & u_2<1,\ v_2<1, \quad   u_2+v_2\ge 1,
\quad 0\le u_3<v_3<1\,.
\end{array}
\end{equation}
By direct computation, we have
$$
\Delta_0= u_2 + u_3 + v_1 + v_2-1.
$$

\medskip
Note that the area of the triangle with vertices $P_i,P_j,P_k$
will not change if we replace $P_i$ by $P_i + d(P_j-P_k)$ for any
$d \in \IR$. Thus,
$\Delta(1,3,5)$ will not be affected and
$\Delta(2,4,6)$  will not
change under the following transformations:
\begin{enumerate}
\item $(u_1,v_1,u_2,v_2,u_3,v_3)\to (u_1+(u_2 + u_3) d,
v_1+ (v_3 - v_2)d, u_2,v_2,u_3,v_3)$,

\item $(u_1,v_1,u_2,v_2,u_3,v_3)\to (u_1,v_1,u_2+(u_1 + u_3) d,v_2-
(v_1 + v_3) d,u_3,v_3)$,

\item $(u_1,v_1,u_2,v_2,u_3,v_3)\to (u_1,v_1,u_2,v_2,u_3+(u_1 - u_2) d,
v_3+ (v_1 + v_2) d)$
\end{enumerate}
For $(i,j,k) \neq (1,3,5)$ and $(2,4,6) $,
$\Delta(i,j,k)<1 $ will hold for sufficiently small $d>0$, whereas
$\Delta_0$ will change to
\begin{enumerate}
\item $\Delta_0+ (v_3 - v_2) d$,

\item $\Delta_0+(u_1+u_3-v_1 - v_3) d$,

\item $\Delta_0+(u_1 -u_2) d$,

\end{enumerate}
respectively. By the maximality of $\Delta_0$, we must have
$$v_2-v_3=(u_1+u_3-v_1 - v_3)=(u_1 -u_2)=0\,,  $$
which gives
$$u_1=u_2,\ v_1 = u_2 + u_3 - v_3,\ v_2 = v_3\,.$$
Substituting into $\Delta(2,4,6)$, we have
$$\Delta(2,4,6)=(u_2+u_3)^2\le 1\ \Rightarrow\ (u_2+u_3)\le 1\,.$$
Substituting into $\Delta_0$, we have
$$\Delta_0=  2 u_2 + 2 u_3-1\le 1,$$
which is the desired contradiction. \qed

\bigskip

By Lemma \ref{boundary}, we can assume that the largest triangle $\Delta$ in the hexagon $\cH$   has at least one side in common with $\cH$. We consider two cases.

\noindent{\bf Case 1} $\Delta$    has   two sides in common with $\cH$.
Then we may
assume that $\Delta$ is the  triangle with vertices $P_1,P_2,P_3$.
Using an affine transformation, we may assume that $P_1=(0,\ 0)$, $P_2=(1,\ 0)$ and
$P_3=(0,\ 1)$. Then  $P_4$, $P_5$ and $P_6$ have to lie inside the triangle with vertices, $(0,\ 0) ,\ (1,\ 1)$ and
$(0,\ 1)$. Therefore, $\cH$ has area less than or equal to $1$, a contradiction.

\bigskip
\noindent{\bf Case 2} $\Delta$   has   one side in common with $\cH$.
Then we may
assume that $\Delta$ is the     triangle with vertices $P_1,P_2,P_4$.

Using an affine transformation, we may assume that $P_1=(0,0)$, $P_2=(1,0)$ and
$P_4=(0,1)$. Let $P_3=(u_1,v_1) $, $P_5=(-u_2,v_2)$ and $P_6=(-u_3,v_3)$, where
$u_1, u_2, u_3,  v_1, v_2, v_3 \ge 0.$
So, we have a hexagon with vertices
$(0,0),\ (1,0),\ (u_1,v_1),\ (0,1),\   (-u_2,v_2),\ (-u_3,v_3)$.
Since the hexagon is convex, we have
\begin{equation}\label{convex2}
u_1+v_1\ge 1,\ v_2\ge v_3,\ u_3 v_2 \ge u_2 v_3,\ \mbox{ and }
u_3 v_2 - u_2 v_3 \ge u_3 - u_2\end{equation}
Let
$$
\tilde A  =\left[\begin{array}{cccccc}
1 & 1 & 1 & 1 & 1     & 1  \\
0 & 1 &u_1  & 0&-u_2  &-u_3\\
0 & 0 &v_1 & 1 &v_2   &v_3 \\
\end{array}\right].
$$
Then  $|\tilde \Delta(i,j,k)|$ is the determinant of the
submatrix of $\tilde A$ lying in columns $i,j,k$, and assume that

\begin{equation}\label{124}
\tilde \Delta(1,2,4)=1 ,\quad    \mbox{ and }
\quad \tilde \Delta(i,j,k)\le 1\ \mbox{ for all } 1\le i<j<k \le 6
\end{equation}
It follows from   (\ref{124}) that

\medskip
(a) \  $(u_1,v_1)$ lies in the triangle with vertices $(1,0), (1,1), (0,1)$. Equivalently, $0\le 1-u_1\le v_1\le 1$.

\medskip
(b) \ $(-u_2,v_2)$ and $(-u_3,v_3)$ lie  in the triangle with vertices
$(0,0), (0,1), (-1,1)$. Equivalently,
$$0\le u_2\le v_2\le 1\quad\hbox{ and } \quad 0\le u_3\le v_3\le 1.$$
Let
$$g(u_1,v_1,u_2,v_2,u_3,v_3)
=  \tilde \Delta(2,3,4) + \tilde \Delta(1,4,5)  + \tilde \Delta(1,5,6)
=u_1 + u_2 + v_1 + u_3 v_2 - u_2 v_3-1 \,.$$
Suppose $g$ attains a maximum $M$ at $(u_1,v_1,u_2,v_2,u_3,v_3)$ subject to
the constraints (\ref{convex2}) and (\ref{124}). We are going to show that
\begin{equation}\label{fmax1}M\le \dfrac{5}{4}
\end{equation}

\begin{lemma}\label{max1}  Suppose  $(u_1,v_1,u_2,v_2,u_3,v_3)$
satisfy (a) and (b) such that
$g(u_1,v_1,u_2,v_2,u_3,v_3)\ge\dfrac{5}{4} $. Then
$$u_1+v_1\ge \dfrac{5}{4} ,\quad  v_2\ge \dfrac{1}{4}.$$
\end{lemma}

\it Proof. \rm Suppose at some  $(u_1,v_1,u_2,v_2,u_3,v_3)$ satisfying (a) and (b), $g(u_1,v_1,u_2,v_2,u_3,v_3)\ge\dfrac{5}{4} $.
Then
$$\begin{array}{rcl}\dfrac{5}{4}&\le &u_1 + u_2 + v_1 + u_3 v_2 - u_2 v_3-1\\&\\
&= &u_1+ v_1 -1+ u_2(1-v_2+u_3)  + (v_2 - u_2) u_3\\&\\
&\le &(u_1+ v_1 -1)+u_2 + (v_2 - u_2)\\&\\
&= &(u_1+ v_1 -1)+ v_2.\end{array}$$
Since $(u_1+ v_1 -1), v_2\le 1$, the result follows.  \qed

Let us focus on the following constraints.

\medskip
(c) $ \tilde \Delta(1,3,5) =u_ 2 v_ 1 + u_ 1 v_ 2  \le 1$,

\medskip
(d) $  \tilde \Delta(1,3,6)    =  u_ 3 v_ 1 + u_ 1 v_ 3   \le 1$,

\medskip
(e) $  \tilde \Delta(2,3,5)=v_1 - v_2+ u_ 2 v_1  + u_1 v_ 2   \le 1$,

\medskip
(f) $  \tilde \Delta(2,3,6)=v_1 - v_3+ u_ 3 v_1  + u_1 v_3  \le 1$,

\medskip\noindent
Consider the maximization problems under the following constraints:

  \begin{enumerate}
   \item $M_1=$ maximum of $g$ under the constraints $v_1\le v_3 $,   (a), (b), (c), (d) and   (\ref{convex2}).

  \item $M_2=$ maximum of $g$ under the constraints $v_3\le v_1\le v_2$,  (a), (b), (c)   and  (f).

%{\marginpar{\tiny We omit (13) in $M_2$ on purpose to simplify the proof later.}}

      \item $M_3=$ maximum of $g$ under the constraints $v_2\le v_1$, (a), (b), (f) and   (\ref{convex2}).
 \end{enumerate}

Because $v_3\le v_2$, we have $M\le \max\{ M_1,\ M_2, \ M_3\}$.
So (\ref{fmax1}) will follow from the following.

  \begin{proposition}\label{M123}
   $M_1,\ M_3\le M_2 \le \dfrac{5}{4}$.
  \end{proposition}

  \it Proof. \rm First we show that
 $M_1, M_3 \le\max\left\{ M_2,\ \dfrac{5}{4}\right\} $.  Let
  $$\ \begin{array}{rl}
   g_1(u_1,v_1,u_2,v_2,u_3,v_3)&
  =\tilde \Delta(1,3,5)= u_ 2 v_ 1  + u_ 1 v_2   \\&\\
   g_2(u_1,v_1,u_2,v_2,u_3,v_3)&=\tilde \Delta(1,3,6)=  u_ 3 v_1  + u_1 v_3  \\&\\
  g_3(u_1,v_1,u_2,v_2,u_3,v_3)&=\tilde \Delta(2,3,5)=v_ 1 - v_ 2+ u_ 2 v_ 1  + u_ 1 v_2   \\&\\
  g_4(u_1,v_1,u_2,v_2,u_3,v_3)&=\tilde \Delta(2,3,6)=v_1 - v_3+ u_ 3 v_1  + u_1 v_3\,.
  \end{array}$$
Suppose $M_1$ is attained at $P = (u_1,v_1,u_2,v_2,u_3,v_3)$
satisfying the constraints $v_1\le v_3 $,  (a), (b), (c), (d) and (\ref{convex2}).
Note that
  $$\begin{array}{rcl}
   g_1(u_1- u_3d,v_1+v_3d,u_2,v_2,u_3,v_3)&=&
   g_1(u_1,v_1,u_2,v_2,u_3,v_3)-(u_3v_2-u_2v_3)d \\
   &\le& g_1(u_1,v_1,u_2,v_2,u_3,v_3)\,,\\&&\\
   g_2(u_1 -u_3d,v_1+v_3d,u_2,v_2,u_3,v_3)&=&
   g_2(u_1,v_1,u_2,v_2,u_3,v_3)\,, \\&&\\
   g(u_1 -u_3d,v_1+v_3d,u_2,v_2,u_3,v_3)&=&g(u_1,v_1,u_2,v_2,u_3,v_3)+(v_3-u_3)d
   \\
   &\ge&  g(u_1,v_1,u_2,v_2,u_3,v_3)\,.\end{array}$$
If $v_1<v_3$, then we may let $d = (v_3-v_1)/v_3$ and
replace $(u_1,v_1)$ by $(u_1-u_3d,v_1+v_3d) = (\tilde u_1, v_3)$
with $\tilde u_1 = u_1 - u_3(v_3-v_1)/v_3$.
Then by the fact that $0 \le u_3 \le v_3 \le 1$,

$$\begin{array}{rl}
\tilde u_1&\ge u_1-(v_3-v_1)=u_1+v_1-v_3\ge 1-v_3\ge 0\\&\\
\tilde u_1+v_3&\ge u_1+v_1 \ge 1 \,.\end{array}$$
Thus, this replacement will neither decrease $M_1$ nor violate the constraints
(a), (b), (c), (d), (\ref{convex2}).
In that case, $P$ also satisfies (f). Therefore, $M_1\le M_2$.

\medskip
Suppose $M_3$ is attained at $P= (u_1,v_1,u_2,v_2,u_3,v_3)$
satisfying the constraints $v_2\le v_1$,   (a), (b), (e) and (f). We may assume that $M_3\ge \dfrac{5}{4}$. Then, by Lemma \ref{max1}, $v_2 \ge \dfrac{1}{4}$.
Note that
  $$\begin{array}{rl}
   g_3(u_1+(1+u_2)d,v_1- v_2d,u_2,v_2,u_3,v_3)&=g_3(u_1,v_1,u_2,v_2,u_3,v_3)\,,\\&\\
   g_4(u_1+(1+u_2)d,v_1- v_2d,u_2,v_2,u_3,v_3)&=g_4(u_1,v_1,u_2,v_2,u_3,v_3)-(v_2-v_3+u_3v_2-u_2v_3)d\\&\\
   &\le g_4(u_1,v_1,u_2,v_2,u_3,v_3)\,, \\&\\
  g(u_1+(1+u_2)d,v_1- v_2d,u_2,v_2,u_3,v_3)&=g(u_1,v_1,u_2,v_2,u_3,v_3)+(1+u_2-v_2)d\\&\\
   &\ge g(u_1,v_1,u_2,v_2,u_3,v_3)\,.\end{array}$$
If $v_1>v_2$, we   may let $d = (v_1-v_2)/v_2$ and replace $(u_1,v_1)$ by
$(u_1+(1+u_2)d, v_1-v_2d) = (\hat u_1, v_2)$ so that
$\hat u_1 = u_1+(1+u_2)d$.
Then

$$\begin{array}{rl}\hat u_1 &\ge u_1\ge 0 ,\\&\\
\hat u_1+v_2 &=u_1+\dfrac{(1+u_2)(v_1-v_2)}{v_2}+v_2 \\
& =u_1+v_1 +\dfrac{(1+u_2-v_2)(v_1-v_2)}{v_2}
\ge u_1+v_1
\ge 1.
\end{array}$$
Such a replacement will neither decrease $M_3$ nor violate the constraints
(a), (b), (f), and (\ref{convex2}). In that case,
$P$ also satisfies (c). Therefore, $M_3\le M_2$.

It remains to prove $M_2 \le   \dfrac{5}{4}$.
Note that we have relaxed the constraint (\ref{convex2}) in the
definition of $M_2$ to simplify the arguments in the following.
On the other hand, we cannot use the assumption that $P_1, \dots, P_6$
are the vertices of a convex polygon anymore.
To establish our result, We need one more lemma.

\begin{lemma}\label{max2}
$M_2$ is attained at some
$(u_1,v_1,u_2,v_2,u_3,v_3)$ satisfying one of the following conditions:
\begin{enumerate}
\item $v_1=v_2=v_3$.
\item $\tilde \Delta(1,3,5)=1$, $v_3=u_3$,  $\tilde \Delta(2,3,6)<1$ and
$v_3=v_1$.
\item $\tilde \Delta(1,3,5)=1$,  $v_3=u_3$ and $\tilde \Delta(2,3,6)=1$ .
%{\marginpar{\tiny By omitting (13) in $M_2$ we simplify the proof because we don't %need to check the condition (13) in the following argument.}}
\end{enumerate}
\end{lemma}

\it Proof. \rm Suppose $M_2$ is attained at some $(u_1,v_1,u_2,v_2,u_3,v_3)$
satisfying  $v_3\le v_1\le v_2$,  (a), (b), (c)   and   (f).
If $v_2=v_3$, then $v_1=v_2=v_3$.

Suppose $v_2>v_3$.  We first show that $\tilde \Delta(1,3,5) = 1$.
Assume that $\tilde \Delta(1,3,5)<1$.
Note that
  $$\begin{array}{rl}
   g_1(u_1,v_1,u_2+d,v_2+e,u_3,v_3)&=g_1(u_1,v_1,u_2,v_2,u_3,v_3)+v_1d+  u_1e ,\\&\\
   g_4(u_1,v_1,u_2+d,v_2+e,u_3,v_3)&=g_4(u_1,v_1,u_2,v_2,u_3,v_3)\,,\\&\\
   g(u_1,v_1,u_2+d,v_2+e,u_3,v_3)&=g(u_1,v_1,u_2,v_2,u_3,v_3)+(1-v_3)d+u_3e \,.\end{array}$$
Then we can do the following to increase $g$
to derive a contradiction. (1) If $v_2 < 1$, then take a suitable $d=e>0$.
(2) If $v_2 = 1$, then
$\Delta(1,3,5) = u_1v_2+u_2v_1 < 1$ implies that $u_2 < 1$ as $u_1+v_1 \ge 1$.
We may let $d > 0 = e$.

\medskip
Next, we show that we may assume that $v_3 = u_3$. Note that
  $$\begin{array}{rcl}
   g_1(u_1,v_1,u_2 ,v_2 ,u_3-(1-u_1)d,v_3-v_1d)&=&g_1(u_1,v_1,u_2,v_2,u_3,v_3)\,,\\&\\
   g_4(u_1,v_1,u_2 ,v_2 ,u_3-(1-u_1)d,v_3-v_1d)&=&g_4(u_1,v_1,u_2,v_2,u_3,v_3)\,,\\&\\
   g(u_1,v_1,u_2 ,v_2 ,u_3-(1-u_1)d,v_3-v_1d)&=&g(u_1,v_1,u_2,v_2,u_3,v_3)+(1-v_2)d \,.
   \end{array}$$
Since $u_1+v_1>1$, we may decrease $v_3-u_3$ without decreasing $g$.
Hence, we may assume that $v_3=u_3$.

We further claim that $v_2>u_2$. If it is not true and
$v_2=u_2$. Then $\tilde \Delta(1,3,5)= (v_1+u_1)u_2 =1$, and
$1+u_2 = 1+v_2 \ge u_1 + v_1 = 1/u_2$
so that $1+u_2 \ge 1/u_2 \ge 0$.
Hence  $u_2 \in [(\sqrt 5 - 1)/2,1]$, and
$$g(u_1, \dots, v_3) = 1/u_2 + u_2 - 1 < 5/4 \qquad \hbox{ for }
u_2 \in [(\sqrt 5 -1)/2,1],$$
which is a contradiction.

\medskip
Now, we can show that $\tilde \Delta(2,3,6) = 1$ or $v_3 = v_1$.
Note that
  $$\begin{array}{rcl}
   g_1(u_1,v_1,u_2 ,v_2 ,u_3+d,v_3+d)&=&g_1(u_1,v_1,u_2,v_2,u_3,v_3)\,,\\&\\
   g_4(u_1,v_1,u_2 ,v_2 ,u_3+d,v_3+d)&=&g_4(u_1,v_1,u_2,v_2,u_3,v_3)+(u_1+v_1 - 1 )d,
   \\&\\
   g(u_1,v_1,u_2 ,v_2 ,u_3+d,v_3+d)&=&g(u_1,v_1,u_2,v_2,u_3,v_3)+(v_2-u_2)d \,.
   \end{array}$$
Suppose $\tilde \Delta(2,3,6)<1$.
If  $v_3<v_1$, then we can increase $g$ by choosing $d >0$, a contradiction.
So we have  $v_3=v_1$.
\qed

\medskip\noindent
{\bf Now we can finish the proof of Proposition \ref{M123}.}

Suppose $(u_1,v_1,u_2,v_2,u_3,v_3)$ satisfies  $v_3\le v_1\le v_2$,
(a), (b), (c), (f)  and one of the conditions in Lemma \ref{max2},
we will show that
$ g(u_1,v_1,u_2,v_2,u_3,v_3)\le \dfrac{5}{4}$ according to the three
conditions.

\noindent{\bf Case 2.1} Suppose $v_1=v_2=v_3=v$. Then we have
$$\tilde \Delta(1,3,5)=(u_1 + u_2) v \,, \qquad
\tilde \Delta(2,4,6) =(u_1 + u_3) v \,,$$
$$g(u_1,v_1,u_2,v_2,u_3,v_3) =   u_1 + u_2(1-v) + v   + u_3 v-1.$$
We need to maximize $g(u_1,v_1,u_2,v_2,u_3,v_3)$ subject to the constraints:
$$\begin{array}{rcl}
(u_1 + u_2) v\le 1&\Lra& u_2\le \dfrac{1-u_1}{v},\\&\\
(u_1 + u_3) v\le 1&\Lra& u_3\le \dfrac{1-u_1}{v},\end{array}$$
and
$$\dfrac{5}{4}\le u_1+v_1\le 2,\ 0\le  u_2,\ u_3\le v\le 1.$$
Because
$\(v-\dfrac{1}{2}\)^2\ge 0$, it follows that
$v^2\ge v-\dfrac{1}{4}\ge 1-u_1$, and hence
$1\ge \dfrac{1-u_1}{v^2}$.
Therefore, the maximum of $ g(u_1,v_1,u_2,v_2,u_3,v_3)$ occurs at
$u_2=u_3=\dfrac{1-u_1}{v}$. Then
$$ g(u_1,v_1,u_2,v_2,u_3,v_3)=u_1+v+\dfrac{1-u_1}{v}-1=h(u_1,v).$$
Since $\dfrac{\partial h}{\partial v}=1-\dfrac{1-u_1}{v^2}\ge 0$,
the maximum of $h$ occurs at $v=1$, which gives $h(u_1,1)=1< \dfrac{5}{4}$.

\medskip
\noindent{\bf Case 2.2} Suppose $\tilde \Delta(1,3,5)=1$,   $v_3=u_3 =v_1=v$.
Then we have
$$\tilde \Delta(1,3,5)=u_2 v + u_1 v_2 = 1\Ra u_2 = \dfrac{(1 - u_1 v_2)}{v}\, $$
and
$$ g(u_1,v_1,u_2,v_2,u_3,v_3)=(u_1+v)(1+v_2)+\dfrac{1-u_1v_2}{v}-2=k(u_1,v_2,v).$$
So we want to maximize $k(u_1,v_2,v)$ subject to
$$\tilde \Delta(2,3,6)= v(u_1+v)\le 1,\quad
\dfrac{1}4\le \dfrac{5}4-u_1\le v \le v_2\le 1,\quad
\dfrac{1-u_1v_2}{v}\le v_2.$$
Equivalently,
$$\dfrac{1}4\le \dfrac{5}4-u_1\le v \le \dfrac{1}{v+u_1}\le v_2\le 1.
$$
Note that $\dfrac{\partial k}{\partial v_2}=v-u_1\(\dfrac{1}{v}-1\)$.

Suppose $\dfrac{\partial k}{\partial v_2}\ge 0$, i.e., $u_1\le \dfrac{v^2}{1-v}$.
Then the maximum of
$k$ occurs at $v_2=1$ so that
$$k(u_1,1,v)=2 u_1 + \dfrac{(1 - u_1)}{v} + 2 v-2.$$
Elementary calculus shows that
the maximum of $2 u_1 + \dfrac{(1 - u_1)}{v} + 2 v-2$ with
$$\dfrac{1}4\le \dfrac{5}4-u_1\le v\le \dfrac{1}{v+u_1}\le 1,\quad u_1
\le \dfrac{v^2}{1-v}$$
occurs at $v=\dfrac{2}3,\ u_1
=\dfrac{5}6$ and $k(\dfrac{5}6,1,\dfrac{2}3)=\dfrac{5}4$.

Suppose $\dfrac{\partial k}{\partial v_2} < 0$, i.e., $u_1 < \dfrac{v^2}{1-v}$.
Then the maximum of $k$ occurs at
$v_2=\dfrac{1}{(u_1+v)}$
so that
$$k(u_1,\dfrac{1}{(u_1+v)},v)=u_1+v+\dfrac{1}{(u_1+v)}-1.$$
Direct calculation shows
that the maximum of $u_1+v+\dfrac{1}{(u_1+v)}-1$ in
$$\dfrac{1}4\le \dfrac{5}4-u_1
\le v\le \dfrac{1}{v+u_1}\le 1,\ u_1\ge \dfrac{v^2}{1-v}$$
occurs at $u_1=1,\ v=\dfrac{\sqrt{5}-1}2$, which gives $v_2=\dfrac{\sqrt{5}-1}2$ and
$k(1, \dfrac{\sqrt{5}-1}2 , \dfrac{\sqrt{5}-1}2 )=\sqrt{5}-1<\dfrac{5}4$.

\medskip
\noindent{\bf Case 2.3}  $\tilde \Delta(1,3,5)=1$, $\tilde \Delta(2,3,6)=1$  and    $v_3=u_3$.
Then we have
$$u_1 = \dfrac{(1 - v_1 + v_3 - v_1 v_3)}{v_3},\ u_2 =\dfrac{v_1 v_2 + v_3 +v_1 v_2 v_3-
  v_2 -   v_2 v_3 }{ v_1 v_3},$$
and
$$ g(u_1,v_1,u_2,v_2,u_3,v_3)=\dfrac{(1 - v_1) (v_1 - v_2)
+ v_3 + (v_2-1 ) v_3^2}{v_1 v_3}=\ell(v_1, v_2, v_3).$$
So we want to maximize $\ell(v_1, v_2, v_3)$ subject to
$$\dfrac{1}4\le
\dfrac{5}4-\dfrac{(1 - v_1 + v_3 - v_1 v_3)}{v_3}
\le v_1\le v_2\le 1,\quad
\dfrac{v_1 v_2 + v_3 +v_1 v_2 v_3-
v_2 -   v_2 v_3 }{ v_1 v_3}\le v_2\le 1.$$
Equivalently,
$$v_1\le  v_2\le   1,\quad \dfrac{1}{1+v_3}\le v_1\le \dfrac{4-v_3}4.
$$
Note that $\dfrac{\partial \ell}{\partial v_2}= \dfrac{v_1+v_3^2-1}{v_1v_3}$.

Suppose $v_1+v_3^2\ge 1$. The
maximum of $\ell$ occurs at $v_2=1$ so that
$\ell(v_1,1,v_3)=  \dfrac{v_3-(1 - v_1)^2}{v_1v_3} $.
Direct calculation shows
that the maximum of $\dfrac{v_3-(1 - v_1)^2}{v_1v_3}$  with
$$v_1\le  v_2\le   1,\quad
\dfrac{1}{1+v_3}\le v_1\le \dfrac{4-v_3}4,\quad
v_1+v_3^2\ge 1$$
occurs at $v_1= v_3=\dfrac{2}3$ and $\ell(\dfrac{2}3,1,\dfrac{2}3)=\dfrac{5}4$.

Suppose $v_1+v_3^2 < 1$. The maximum of $k$ occurs at $v_2=v_1$ so that
$\ell(v_1,v_1,v_3)= \dfrac{1-(1-v_1)v_3}{v_1}$. Direct calculation
shows that the maximum of $\dfrac{1-(1-v_1)v_3}{v_1}$ in
$$v_1\le  v_2\le   1,\ \dfrac{1}{1+v_3}\le v_1\le \dfrac{4-v_3}4,\ v_1+v_3^2\le 1$$
occurs at $v_1= \dfrac{2}3,\ v_3=\dfrac{1}{2}$ and $\ell(\dfrac{2}3, \dfrac{2}3,\dfrac{1}{2})=\dfrac{5}4$. \qed

\bigskip
\noindent
{\bf Remarks} Several comments related to Proposition \ref{Hexagon} are in order.

\begin{enumerate}

\item The proof of Proposition \ref{Hexagon} is direct but quite lengthy.
A shorter proof is desirable.

\item
One might expect that a symmetry argument can be used to
show that the solution of Proposition \ref{Hexagon}
is attained at a regular hexagon by a suitable
affine transform when $n = 6$, but it is not the case as shown by
our result.

\item One may generalize Proposition \ref{Hexagon} and determine
the optimal bound of the ratio between the area
of an $n$-sided convex polygon and the area of an maximal
$m$-sided polygon in it for $m < n$.

\end{enumerate}

\bigskip\noindent
{\bf Acknowledgment}

The research of Li and Poon
was supported by USA NSF, and HK RGC.
Li  was an honorary professor of the Shanghai University,
and an honorary professor of the University of Hong Kong.
The research of Wang was done while he was visiting the College of William and
Mary during the academic
year 2013-14 under the support of China Scholarship Council.

\end{document}